\documentclass{amsart}
\usepackage{graphicx}
\usepackage{psfrag}
\usepackage[active]{srcltx}
\usepackage{amssymb}
\usepackage{amscd}

\newtheorem{theorem}{Theorem}[section]
\newtheorem{cor}[theorem]{Corollary}
\newtheorem{lm}[theorem]{Lemma}

\theoremstyle{definition}
\newtheorem{df}[theorem]{Definition}

\theoremstyle{remark}

\numberwithin{equation}{section}

\newif\ifShowLabels
\ShowLabelstrue
\newdimen\theight
\newcommand\TeXref[1]{%
     \leavevmode\vadjust{\setbox0=\hbox{{\tt
               \quad\quad #1}}%
                    \theight=\ht0
                         \advance\theight by \dp0
                              \advance\theight by \lineskip
                                   \kern -\theight \vbox to
                                             \theight{\rightline{\rlap{\box0}}%
                                                   \vss}%
                                                         }}%
\newcommand{\labelp}[1]{\label{#1}%
    \ifShowLabels \TeXref{{}} \fi}

\renewcommand\L{H}

\newcommand\G[1]{G_{#1}}
\newcommand\conv{{}^{\scriptstyle\smallsmile}}
\newcommand\ident{1\textnormal{\rq}}
\newcommand{\refL}[1]{Lemma~\ref{L:#1}}
\newcommand{\refD}[1]{Definition~\ref{D:#1}}
\newcommand{\refC}[1]{Corollary~\ref{C:#1}}
\newcommand{\refT}[1]{Theorem~\ref{T:#1}}
\newcommand{\refS}[1]{Section~\ref{S:#1}}

\newcommand{\refF}[1]{Figure~\ref{F:#1}}
\newcommand{\mc}[1]{\mathcal{#1}}
\newcommand\cra[2]{\mathfrak{#1}\myshortspace[\mathcal{#2}\,]}

\newcommand{\myshortspace}{{\hspace*{.1pt}}}
\newcommand\f[1]{{\mathfrak {#1}}}

\newcommand\smbcomma{\,,}
\newcommand{\comma}{\textnormal{,}\ }
\newcommand\per{\textnormal{\myspace.\ }}
\newcommand{\myspace}{{\hspace*{.5pt}}}
\newcommand\po{\textnormal{.}\ }     % point
\newcommand\rcomma{\textnormal{\myshortspace,}}
\newcommand\yoz{\wy;1;\wz}
\newcommand\wy{y}
\newcommand\wz{z}
\newcommand\xoy{\wx;1;\wy}
\newcommand\wx{x}
\newcommand\grp[1]{G_{#1}}
\newcommand\scir{\raise2pt\hbox{$\,\scriptscriptstyle\circ\,$}}
\newcommand\mo{^{-1}}
\newcommand\pair[2]{(#1,#2)}
\newcommand\mce{\mathcal{E}}
\newcommand\h[1]{H_{#1}}
\newcommand\xy{{xy}}
\renewcommand\k[1]{K_{#1}}
\newcommand\xz{{xz}}
\newcommand\tbigcup{\textstyle \bigcup}
\newcommand\seq{\subseteq}
\newcommand\refEq[1]{(\ref{Eq:#1})}
\newcommand\yz{{yz}}
\newcommand\vphi[1]{\vph_{#1}}
\newcommand\vph{\varphi}
\newcommand\vp{\varphi}
\newcommand\hvphs[1]{\hat{\vph}_{#1}}
\newcommand\e[1]{e_{#1}}
\newcommand\yx{{yx}}
\newcommand\vphih[1]{{\hat\vph}_{#1}}
\newcommand\rp{\!\mid\!}
\newcommand\xx{{xx}}
\newcommand\suml{\sum\limits}
\newcommand\R{R}
\newcommand\xoz{\wx;1;\wz}
\newcommand\tprod{\textstyle\prod}
\newcommand\tsum{\textstyle\sum}
\newcommand\xox{\wx;1;\wx}
\newcommand\yox{\wy;1;\wx}
\newcommand\cs[2]{{#1}_{#2}}
\newcommand\di{0\textnormal{{\rq}}}

\begin{document}

\title[Measurable relation algebras with cyclic groups]{A representation
theorem for measurable relation algebras with cyclic groups}
\author{Hajnal Andr\'eka and Steven Givant}
\address{Hajnal Andr\'eka\\Alfr\'ed R\'enyi Institute of
Mathematics\\Hungarian Academy of Sciences\\Re\'altanoda utca 13-15\\
Budapest\\ 1053 Hungary}\email{andreka.hajnal@renyi.mta.hu}
\address{Steven Givant\\Mills College\\5000 MacArthur Boulevard,
Oakland, CA 94613}\email{givant@mills.edu}
\thanks{This research was
partially supported  by Mills College and  the Hungarian National
Foundation for Scientific Research, Grants T30314 and T35192.}
\begin{abstract} A relation algebra is measurable if the identity element is a sum of
atoms, and the square $x;1;x$ of each subidentity atom $x$ is a sum
of non-zero functional elements. These functional elements form a
group $\G x$. We prove that a measurable relation algebra in which
the groups $\G x$ are all finite and cyclic is completely
representable. A structural description of these algebras is also
given.
\end{abstract}

\maketitle

\section{Introduction}\labelp{S:sec1}

Relation algebras were defined by Tarski as a class of abstract
algebras satisfying ten equations, to serve as an algebraic
counterpart of logic. They are Boolean algebras with operators where
the extra-Boolean operators $;,\conv,\ident$ form an involuted
monoid (satisfying some further equations). In particular, the
complex algebra of a group is a relation algebra. Algebras of binary
relations with standard set theoretic Boolean operations and
relation-composition, inverse of a relation and set theoretic
identity relation as the extra-Boolean operations are called
\emph{set relation algebras}. These are relation algebras, and a
relation algebra isomorphic to a set relation algebra is called
\emph{representable}. In particular, the function assigning the
Cayley-representation to each element of a group extends to a
representation of the complex algebra of the group.

Measurable relation algebras were introduced in \cite{ga02}. In a
relation algebra, a subidentity atom $x$ is defined to be measurable
if $x;1;x$ is the sum of functional elements and a relation algebra
is \emph{measurable} if $\ident$ is the sum of measurable atoms.
(Atoms below the identity $\ident$ are called subidentity atoms, and
an element $f$ is called functional if $f\conv ; f\le \ident$.)
Measurable relation algebras are closely related to groups. They are
put together from various groups as follows. The non-zero functional
elements below $x;1;x$ are atoms and they form a group $G_x$ with
identity element $x$ under the operations of $;$ and $\conv$, and
the atoms below $x;1;y$ specify isomorphisms between quotient groups
$G_x\slash H_{xy}$ and $G_y\slash K_{xy}$, when $x,y$ are distinct
subidentity atoms (see \cite[p.53]{ga02}, \cite{giv1} and
\refS{sec4} in this paper).

A group pair is defined as a system of groups $G_x$ with normal
subgroups $H_{xy}$ and $K_{xy}$, and isomorphisms $\varphi_{xy}$
between quotient groups $G_x/H_{xy}$ and $G_y/K_{xy}$, for $x,y\in
I$. When the isomorphisms are linked by certain natural conditions,
we can put together the Cayley-representations of the various groups
occurring in the group pair $\mc F$ to get a set relation algebra
$\cra G F$. A \emph{group frame} is a group pair that satisfies the
natural conditions, given in this paper as \refD{cosfra}. The
algebras $\cra G F$ constructed from group frames are called
(generalized full) \emph{group relation algebras}. Group relation
algebras are all measurable, atomic, complete set algebras, with all
suprema being union in them (i.e., completely representable). For
examples and illustrations see \cite{ga02} and \cite{giv1}.

An immediate question is whether group relation algebras exhaust the
examples of all atomic complete measurable relation algebras.

This paper proves that indeed this is the case when all the groups
$G_x$ in the measurable relation algebra are finite and cyclic
(Representation \refT{rep}). Even the conditions of being atomic and
complete can be omitted: all measurable relation algebras $\f A$
with finite cyclic groups are atomic and \emph{essentially}
isomorphic to group relation algebras. This latter means that the
completion (the minimal complete extension) of $\f A$ is isomorphic
to a group relation algebra. The passage to the completion of $\f A$
does not change the structure of $\f A$, it only fills in any
missing infinite sums that are needed in order to obtain isomorphism
with the necessarily complete full group relation algebra. We note
that an atomic measurable relation algebra is essentially isomorphic
to a group relation algebra if and only if it is completely
representable (\cite[Theorems 7.4, 7.6]{givand3}).

In the case when all the groups are
cyclic, the group frame conditions considerably simplify, see
\refD{index} here. This fact gives measurable relation algebras with
finite cyclic groups an especially clear structural description.

The algebras that come up in Theorem 4.30 of
J\'onsson-Tarski~\cite{jt52} are all measurable with associated
groups being one-element. Hence, the hard direction (ii)
$\Rightarrow$ (i) of \cite[Theorem 4.30]{jt52} follows from our
\refT{rep}.
Algebras with all the associated groups being cyclic of order one or
two also have come up in the literature. We show, in
Lemma~\ref{pair-lem}, that among atomic relation algebras they are
exactly the pair-dense algebras of Maddux~\cite{ma91}. This gives a
possibility for giving a new proof for \cite[Theorem 51]{ma91}.

An extension of Representation \refT{rep} is known. If $\f A$ is a
measurable relation algebra in which the group $\G x$ is a product
of two finite cyclic groups for each subidentity atom $x$, then $\f
A$ is essentially isomorphic to a group relation algebra (this is
announced in \cite[Theorem 6]{ga02}). The proof of this extended
theorem (due to the authors) is complicated, and will not be given
here.

The theorem cannot be extended to the case in which the groups $\G
x$ are allowed to be products of three finite cyclic groups. Indeed,
an example is presented in \cite[pp.56-59]{ga02}, and proved to be
nonrepresentable in \cite[Theorem 5.2]{andgiv1}, of a finite
measurable relation algebra with five measurable atoms $x$ such
that, the group $\G x$ is isomorphic to $\mathbb Z_2\times\mathbb
Z_2\times\mathbb Z_2$, for all $x$.

There are several other instances in which an atomic measurable
relation algebra is essentially isomorphic to a group relation
algebra. This proves to be the case, for example, if there are at
most four measurable atoms in all.  It is also the case if, for
every measurable atom $x$, every chain of normal subgroups of $\G x$
has length at most three (so that there cannot be normal subgroups
$H$ and $K$ of $\G x$ such that $\{x\}\subsetneq H\subsetneq
K\subsetneq \G x$).  Another case is when the normal subgroups
$K_{xy}$ and $H_{yz}$ are always the same. See
\cite[pp.56-59]{ga02}.

As an application of some of these ideas, let us paraphrase Maddux's
terminology by calling a relation algebra $n$-\emph{dense} if it is
measurable, and if, for each measurable atom $x$, the group $\G x$
has cardinality at most $n$.  Every group of cardinality at most $7$
is either cyclic, the product of two cyclic groups, or has no normal
subgroup chains of length more than three.  As a result, we obtain
that every $7$-dense relation algebra is representable, but there
exist $8$-dense relation algebras (with five measurable atoms) that
are not representable.

The above results are summarized in \cite{ga02} (without proofs),
along with detailed motivation and many illustrations. Readers who
wish to learn more about the subject of relation algebras are
recommended to look at the books by Hirsch-Hodkinson\,\cite{hh02},
Maddux\,\cite{ma06}, or Givant\,\cite{giv17}, \cite{giv18}.

\newcommand\kap[2]{m_{{#1}{#2}}}
\newcommand\al{a}
\newcommand\rr[3]{R_{{#1}{#2},{#3}}}
\newcommand\hl[2]{H_{{#1}{#2}}}
\newcommand\hr[2]{K_{{#1}{#2}}}
\newcommand\gl[2]{G_{#1}/H_{{#1}{#2}}}
\newcommand\gr[2]{G_{#2}/K_{{#1}{#2}}}

The remainder of the paper is divided into four sections.
\refS{sec2} reviews  the necessary relation algebraic background for
reading the paper.  Throughout the paper, there are  references to
the earlier papers \cite{giv1}, \cite{andgiv1},  and \cite{givand3};
the results needed from  those papers are explained as they are
encountered. In  \refS{sec3}, a characterization is given, in terms
of a system of simple invariants, of when it is possible to
construct a full group relation algebra from a given system of
mutually disjoint, finite, cyclic groups.  \refS{sec4} discusses the
notion of a regular element---a kind of generalization of the notion
of an atom---and of the index of such an element.  Some important
properties of  indices are formulated and proved.  Finally,
\refS{sec5} is devoted to a proof of the main theorem of the paper,
namely to Representation \refT{rep}.

\section{Relation Algebras}\labelp{S:sec2}
In the next few sections, most of the calculations will involve the
arithmetic of relation algebras.  This section provides a  review of
the essential results that will be needed.

 A relation algebra is an algebra   of the form
\[\f A=( A\smbcomma +\smbcomma
-\smbcomma ;\smbcomma\,\conv\smbcomma\ident)\comma  \]  where
$\,+\,$ and $\,;\,$ are binary operations called \textit{addition}
and \textit{relative multiplication}, while $\,-\,$ and $\,\conv\,$
are unary operations called \textit{complement} and
\textit{converse}, and $\ident$ is a distinguished constant called
the \textit{identity element}, such that ten equational axioms hold
in $\f A$.  The exact nature of these axioms is not important for
the present discussion.  Further operations and relations such as
Boolean multiplication $\,\cdot\,$ and the Boolean partial  order
relation $\,\le\,$ are defined in the standard way. The following
laws, which are either members of Tarski's ten axioms or are
derivable from them, play a role in this paper. For their
derivation, see e.g., \cite{giv17}.

\begin{lm}\labelp{L:laws} If $\f A=( A\smbcomma +\smbcomma
-\smbcomma ;\smbcomma\,\conv\smbcomma\ident)$ is a relation
algebra\comma then $(A\smbcomma +\smbcomma-)$ is a Boolean
algebra\comma and the operation of converse is an automorphism of
this Boolean algebra\per In particular, the following laws hold\per
\begin{enumerate}
\item[(i)] $(a+b)\conv=a\conv + b\conv$\per
\item[(ii)] $(a \cdot b)\conv = a\conv \cdot
b\conv$\po
\item[(iii)]$a \leq b$ if and only if $a\conv \leq
b\conv$\po
\item[(iv)] $a$ is an atom if and only if $a\conv$ is an atom\po
\item[(v)]   $x\conv=x$ and $x;x=x$ whenever $x$ is a subidentity element\po
\end{enumerate}
\end{lm}
Properties (i) and (iii) are called the \emph{distributive} and
\emph{monotony laws} for converse. \goodbreak

\begin{lm}\labelp{L:laws.1}
 \begin{enumerate}
\item[(i)]
$r;(s;t)=(r;s);t$\per
\item[(ii)]
$r;\ident=r$\per \item[(iii)] $r\conv\conv=r$\per
\item[(iv)]
$(r;s)\conv=s\conv;r\conv$\per \item[(v)] $(r+s);t=r;t +s;t$\per
 \item[(vi)] If $a \leq b$ and $c\leq d$\rcomma\ then
$a;c \leq b;d$\po
\end{enumerate}
\end{lm}
These properties are commonly referred to by the following names:
the \textit{associative law for relative multiplication},   the
(right-hand) \textit{identity law for relative multiplication},  the
\textit{first involution law},   the \textit{second involution law},
the (right-hand) \textit{distributive law for relative
multiplication}, and the \emph{monotony law for relative
multiplication}. An element $x$ in $\f A$ is called a
\textit{subidentity element} if it is below the identity element, in
symbols $x\le \ident$. Whenever parentheses indicating the order of
performing operations are  lacking, it is understood that unary
operations have priority over binary operations, and multiplications
have priority over addition.

A \textit{square}   is an element of the form $x;1;x$ for some
subidentity element $x$, and a \textit{rectangle} is an element of
the form $x;1;y$ for some subidentity elements $x$ and $y$. The
elements $x$ and $y$ are sometimes referred to as the \textit{sides}
of the rectangle\per

\begin{lm}\labelp{L:square} Let $x,y,z,w$ be  subidentity
elements\po\
\begin{itemize}
\item[(i)] $(x;1;y)\conv = y;1;x$\po
\item[(ii)] $(x;1;y);b\leq x;1;z$  for every $b\leq \yoz$\comma  and equality holds whenever
$x$\comma $y$\comma and $z$ are atoms, and $x;1;y$ and $b$ are both non-zero\per
\item[(iii)] If $x$ and $y$ are subidentity atoms, and if $0\le b\le\xoy$, then
$x;b=b;y=b$\per
\end{itemize}
\end{lm}

\section{Group relation algebras with finite cyclic groups}\labelp{S:sec3}

\newcommand\hvps[1]{\cc{\hat\vp}{#1}}
\renewcommand\al{a}
\newcommand\map[2]{m_{{#1}{#2}}}
\renewcommand\rr[3]{R_{{#1}{#2},{#3}}}
\renewcommand\hl[2]{H_{{#1}{#2}}}
\renewcommand\hr[2]{K_{{#1}{#2}}}
\newcommand\zn[1]{{\mathbb Z}_{n_{#1}}}
\renewcommand\gl[2]{G_{#1}/H_{{#1}{#2}}}
\newcommand\cc[2]{{#1}_{#2}}
\renewcommand\gr[2]{G_{#2}/K_{{#1}{#2}}}

Let $G=\langle\grp x:x\in I\,\rangle$ be a system of pairwise
disjoint, finite, cyclic groups $(\G x\smbcomma \scir\smbcomma
{}\mo\smbcomma e_x)$,   and \[m=\langle\map  xy:\pair xy\in \mc
E\rangle\comma\] a system of positive integers indexed by  an
equivalence relation $\mce$ on the index set $I$. We may assume that
for each $x$ in $I$, the cyclic group $\G x$ is a copy of the cyclic
group $\zn x$ of integers modulo $\cc nx$ for some positive integer
$\cc nx$. We shall usually act and write as if the two groups were
identical, although technically it is important to pass to a copy of
$\zn x$ in order to achieve the assumed disjointness of the groups
in the system  $G$. The greatest common divisor of two numbers $m,n$
is denoted by $\gcd(m,n)$\per For the following definition see also
\cite[p.51]{ga02}.

\begin{df} \label{D:index} The system of indices
\[m=\langle\map  xy:\pair xy\in \mc E\rangle\comma\]
is said to satisfy the \emph{index conditions} if the following conditions  hold.
\begin{enumerate}
\item[(i)] $ \map  x y$ is a common divisor of the orders of $\grp x$
and $\grp y$\per
\item[(ii)] $ \map  x x$ is equal to the order of $\grp x$\per
\item[(iii)] $ \map  y x = \map  x y$\per
\item[(iv)] $\gcd( \map  x y,\map  y z) =  \gcd( \map  x y,\map  x
z)=\gcd( \map  x z,\map  y z)$\per
\end{enumerate}
\end{df}

Assume  $m$ does satisfy the index conditions, and write
 \[d=\gcd\pair {\map  xy}{\map  xz}=\gcd\pair {\map  xy}{\map  yz}\per\]
This equation holds because of index condition (iv). For each pair
$\pair xy$ in $\mc E$, take $\h\xy$ and $\k \xy$  be the respective
subgroups of $\G x$ and $\G y$ that consist  of the multiples of
$\map  xy$. This definition makes sense because of the first index
condition. It follows that $\h\xy$ and $\k\xy$ have index $\map  xy$
in $\G x$ and $\G y$ respectively, that is to say, they each have
$\map  xy$ cosets, in symbols,
\[\map  xy=|\G x/\h\xy|=|\G x/\k\xy|\per\] Notice that these
subgroups are always normal, since the groups in the system $G$ are
all cyclic, and hence Abelian. Because $\map  xy$ is the index  in
$\G x$ of the subgroup $\h\xy$, this  subgroup must consist of the
multiples of the integer $\map  xy$ modulo $\cc nx$.  In particular,
the cosets of $\h \xy$ are the  sets of the form
\[\h\xy+\ell=\{p\map  xy+\ell:p<\cc nx/\map  xy\}\] for $\ell<\map  xy$.

The composite group \[\h\xy\scir\h\xz =\{h\scir k: h\in\h\xy\text{
and }k\in\h\xz\}\] consists of the multiples of $d$ modulo $\cc nx$,
and the cosets of $\h\xy\scir\h\xz$ are the sets of the form
 \[\h\xy\scir\h\xz+s\] for $0\le s< d$.

\begin{lm} \labelp{L:lm1}
For each integer $s$ with $0\le s<d$\comma \begin{align*}
\h\xy\scir\h\xz+s&=\tbigcup\{\h\xy+\ell:0\le \ell < \map  xy\text{ and }\ell\equiv s\mod d\}\\
&=\tbigcup\{\h\xy+qd +s:q<\map  xy/d\}\per\\
\end{align*}
\end{lm}

\begin{proof} Each non-negative integer $\ell <\map  xy$ can be written in
one and only one way in the form
\[\ell=qd+s\] for some integers $s$ and $q$ satisfying $0\le s<d$ and
$0\le q<\map  xy/d$, by the division algorithm for integers, so the second
equality of the lemma is clear.

Observe that $\h\xy$ is included in the composite group $\h\xy\scir\h\xz$.
Also, $d$ gen\-er\-ates the composite group, so $qd$ is in the composite group
for every integer $q$ with $0\le q < \map  xy/d$.  Combine these observations to see that
\[\h\xy + qd\seq\h\xy\scir\h\xz\comma\] and therefore
\begin{equation*}%\tag{1}\labelp{Eq:lm1.1}
  \h\xy +qd+s\seq\h\xy\scir\h\xz +s\comma
\end{equation*} for every $s<d$ and every $q<\map  xy/d$. It follows that
\begin{equation*}\tag{1}\labelp{Eq:lm1.1}
 \tbigcup\{ \h\xy +qd+s:q<\map  xy/d\}\seq\h\xy\scir\h\xz +s\comma
\end{equation*} for every $s<d$.

The cosets that make up the union on the left side of \refEq{lm1.1}
partition $\G x$ as $q$ and $s$ vary, since there are assumed to be
$\map  xy$ such cosets,   one for each $\ell=qd+s<\map  xy$.  Also
the cosets on the right partition $\G x$ as $s$ varies.  It follows
that equality must hold in \refEq{lm1.1}.  In more detail, if $f$
belongs to the right side of \refEq{lm1.1}, then $f\equiv s\mod d$.
Consequently, $f $ cannot belong to any of the cosets $\h\xy +qd+t$
for $0\le t<d$ and $t\neq s$, since the elements in these cosets are
congruent to $t$ modulo $d$.  Thus, $f$ must belong to $\h\xy +qd+s$
for some $q$ with $0\le q<\map  xy/d$.
\end{proof}

One sees in a similar fashion that the subgroup $\k\xy$ has $\cc
ny/\map  xy$ elements and   $\map  xy$ cosets, which  have the form
$\k\xy +\ell$ for $\ell<\map  xy$.  The composite subgroup $\k
\xy\scir\h\yz$ is generated by $d$, and has $\cc ny/d$ elements and
$d$ cosets \[\k\xy\scir\h\yz+s\] for $0\le s<d$.  The proof of the
next lemma is very similar to that of \refL{lm1}, and will therefore
be omitted.

\begin{lm} \labelp{L:lm2}
For each integer $s$ with $0\le s<d$\comma
\begin{align*}
\k\xy\scir\h\yz+s&=\tbigcup\{\k\xy+\ell:0\le \ell < \map  xy\text{ and }\ell\equiv s\mod d\}\\
&=\tbigcup\{\k\xy+qd +s:q<\map  xy/d\}\per
\end{align*}
\end{lm}

Define a mapping $\vphi\xy$ from $\G x/\h \xy$ to $\G y/\k \xy$  by
\[\vphi\xy(\h\xy +\ell)=\k\xy +\ell\]for $0\le \ell<\map  xy$.
The mapping is certainly a bijection, by the preceding remarks,
and it maps the generator $\h\xy +1$ of the quotient group $\G x/\h\xy$
to the generator $\k\xy +1$ of the quotient group $\G y/\k\xy$, so it must
be an isomorphism, as is easy to check directly.  This isomorphism induces
an isomorphism $\hvps \xy$ from $\G x/(\h\xy\scir\h\xz)$ to $\G y/(\k\xy\scir\h\yz)$.

\begin{lm} \labelp{L:lm3}
 $\hvphs\xy(\h\xy\scir \h\xz + s)=\k\xy\scir \h\yz + s$ for   $0\le s< d$\per
\end{lm}
\begin{proof}Use the definition of $\hvphs\xy$, \refL{lm1}, the definition of $\vphi\xy$,
and \refL{lm2} to obtain
\begin{align*}
  \hvps\xy(\h\xy\scir \h\xz + s)&=\vphi\xy[\tbigcup\{\h\xy + qd + s:0\le q<\map  xy/d\}]\\
  &=\tbigcup\{\vphi\xy(\h\xy + qd + s):0\le q<\map  xy/d\}\\
  &=\tbigcup\{\k\xy + qd + s:0\le q<\map  xy/d\}\\
  &=\k\xy\scir\h\yz +s\per
\end{align*}
\end{proof}

In a similar fashion, there is a quotient isomorphism $\vphi\yz$
from $\gl yz $ to $\gr yz$ that is defined by
\[\vphi\yz(\h\yz +\ell)=\k\yz +\ell\] for $0\le \ell<\map  yz$.
This isomorphism, in turn, induces an isomorphism $\hvps \yz$
from $\G y/(\k\xy\scir\h\yz)$ to $\G z(\k\xz\scir\k\yz)$ that satisfies the following lemma.

\begin{lm} \labelp{L:lm4}
 $\hvphs\yz(\k\xy\scir \h\yz + s)=\k\xz\scir \k\yz + s$ for   $0\le s< d$\per
\end{lm}

Finally, there is a quotient isomorphism $\vphi\xz$ from $\gl xz $ to $\gr xz$ that is defined by
\[\vphi\xz(\h\xz +\ell)=\k\xz +\ell\] for $0\le \ell<\map  xz$.  This
isomorphism  induces an isomorphism $\hvps \xz$ from $\G x/(\h\xy\scir\h\xz)$
to $\G z(\k\xz\scir\k\yz)$ that satisfies the following lemma.

\begin{lm} \labelp{L:lm5}
 $\hvphs\xz(\h\xy\scir \h\xz + s)=\k\xz\scir \k\yz + s$ for   $0\le s< d$\per
\end{lm}

The following definition is from \cite[Definition 1]{ga02}, see also
\cite[Definition 4.1]{giv1}).

\begin{df}\labelp{D:cosfra} A \textit{group} \textit{frame}  is a group pair
\[\mc F=(\langle \G x:x\in
I\,\rangle\smbcomma\langle\vph_\xy:\pair x y\in \mc E\,\rangle)\]
satisfying the following  \textit{frame conditions} for all pairs
$\pair xy$ and $\pair yz$ in $\mc E$\per
\begin{enumerate}
\item[(i)]
$\vphi {xx}$ is the identity automorphism of $\G x/\{\e x\}$ for
all $x$\per
\item [(ii)]
$\vphi \yx=\vphi \xy\mo$\per
\item[(iii)] $\vphi \xy[\h \xy\scir\h\xz] =\k
\xy\scir\h \yz$\per
\item[(iv)] $\vphih \xy\rp\vphih \yz=\vphih \xz$\po
\end{enumerate}\qed\end{df}

It is shown in \cite{giv1} that this  definition gives  necessary
and sufficient conditions  for a group pair $\mc F$ to give rise to
a group relation algebra $\cra G F$.
We recall the definition of $\cra G F$ from \cite{ga02},
\cite{giv1}. Suppose that $(H_{xy,\alpha} : \alpha<\kappa_{xy})$ is
a listing of the cosets of $H_{xy}$ in $G_x$. Define
\[ R_{xy,\alpha} = \underset{\gamma<\kappa_{xy}}{\bigcup}\
H_{xy,\gamma}\times\varphi_{xy}(H_{xy,\gamma}\scir
H_{xy,\alpha})\per\] Let $A$ be the set of all binary relations of
form $\bigcup\{ R_{xy,\alpha} : (x,y,\alpha)\in X\}$, where
$X\subseteq\{ (x,y,\alpha) : (x,y)\in \mc E\mbox{ and
}\alpha<\kappa_{xy}\}$\per When $\mc F$ is a group frame, then $A$
is a set of binary relations that is closed under the Boolean
set-theoretic operations, that contains the identity relation on
$\bigcup\{ G_x : x\in I\}$, and that is closed under the operations
of forming the composition of two binary relations and the converse
of a binary relation. The set relation algebra with universe $A$ is
denoted by  $\cra G F$. It is easy to see that each supremum in
$\cra G F$ is indeed a union, so $\cra G F$ is completely
represented.

\begin{theorem}[GCD Theorem]\labelp{T:gcd}
Let $G=\langle \G x:x\in I\rangle$ be a system of mutually
disjoint\comma finite\comma cyclic groups\comma and $\mce$
an equivalence relation on $I$\per For each system
\[m=\langle\map  xy:\pair xy\in\mc E\rangle\]  of positive
integers satisfying the index conditions\comma there exists a system
of quotient isomorphisms $\vp=\langle\vphi\xy:\pair xy\in\mc
E\rangle$ such that the group pair $\mc F=\pair G\vp$ satisfies the
four frame conditions and is therefore a group frame\per The
corresponding group relation algebra $\cra G F$ therefore exists\per
Moreover\comma
\[\map  xy=|\G x/\h\xy |\comma\] where $\h\xy$ is the kernel of $\vphi\xy$\per
\end{theorem}

\begin{proof} Consider, first, frame condition (i).  Index condition (ii)
implies that $\map  xx$ coincides with the cardinality $\cc nx$ of
the group. The subgroups $\h\xx$ and $\k\xx$  consist  of the
multiples of $\map  xx$ in $\G x$, so they must be the trivial
subgroup $\{0\}$.  The definition of $\vphi\xx$ and the natures of
$\h \xx$  and $\k \xx$ imply that
\[\vphi\xx(\{\ell\})=\vphi\xx(\h\xx +\ell)=\k\xx +\ell=\{\ell\}\comma\]
so $\vphi\xx$ is the identity automorphism of $\G x/\{0\}$.
Thus, frame condition (i) holds.

Turn now to frame condition (ii).  The subgroup $\h\yx$ is defined
to be the set of multiples of $\map  yx$ in $\G y$, and the subgroup
$\k\xy$ is defined to be the set of multiples of $\map  xy$ in $\G
y$.  Index condition (ii) says that $\map  yx=\map  xy$, so
$\h\yx=\k\xy$, and similarly, $\k\yx=\h\xy$.   Furthermore,
\[\vphi\yx(\h\xy+\ell)=\k\yx+\ell,\] by the definition of $\vphi\yx$, while
\[\vphi\xy(\k\yx + \ell)=\vphi\xy(\h\xy + \ell)=\k\xy +\ell =\h\yx +\ell\comma\]
by the preceding remarks and the definition of $\vphi\xy$.  Combine
these observations to conclude that $\vphi\yx=\vphi\xy\mo$, which is
what frame condition (ii) asserts.

To verify frame condition (iii), just take $s=0$ in
Lemmas~\ref{L:lm3} and \ref{L:lm4}. In a similar fashion, frame
condition (iv) follows from Lemmas~\ref{L:lm3}--\ref{L:lm5},
because
\begin{multline*} (\hvphs\xy\rp\hvphs\yz)(\h\xy\scir\h\xz
+s)=\hvphs\yz(\hvps\xy(\h\xy\scir\h\xz
+s))\\=\hvphs\yz(\k\xy\scir\h\yz +s)=\k\xz\scir\k\yz +s\comma
 \end{multline*}
 by the definition of the relational composition of two functions, and
 Lemmas~\ref{L:lm3} and \ref{L:lm4}, while
 \[\hvphs\xz(\h\xy\scir\h\xz +s)=\k\xz\scir\k\yz +s\comma\] by \refL{lm5}.
\end{proof}

It is helpful to visualize index conditions (ii)-(iv) by making a
diagram such as  the one in  \refF{fig1}.  Condition (ii) says that
each square in the diagram that is on the line $y=x$ (the identity
relation) carries the same number as the cardinality of the
corresponding group.  In the example given in \refF{fig1}, each such
square is labeled with the same number $6$, because each group is
assumed to have cardinality $6$, but of course in other examples
different groups may have different cardinalities. Condition (iii)
says that the diagram must be symmetric across the line $y=x$. To
check the validity of condition (iv), it must be shown that any two
of any three given  indices $\map rs$, $\map rt$, $\map st$ have the
same greatest common divisor as any other two of the given indices.
This can be checked one column at a time.  Take two numbers that are
in the  $r$th column, $\map rs$ and $\map rt$, and then use either
row $s$ (if $s$ is to the left of $t$ in the column listing) or row
$t$ (if $t$ is to the left of $s$ in the column listing) to locate
$\map st$ or $\map ts$ respectively (it doesn't matter which one
because the two indices must be equal), by going along the row to
the right until the appropriate column is reached. For a concrete
example, observe that in the $u$th column, $\map uv=3$ and $\map
uy=2$.  Since $y$ is to the left of $v$ in the column listing, go to
the $y$th row, and move right to the $v$th column.  The entry there
is $\map vy=1$.  Any two of the three numbers $3$, $2$, and $1$ have
the same greatest common divisor, namely $1$.

\begin{figure}[htb]  \psfrag*{x}[B][B]{$x$} \psfrag*{y}[B][B]{$y$}
\psfrag*{z}[B][B]{$z$} \psfrag*{u}[B][B]{$u$}
 \psfrag*{v}[B][B]{$v$}  \psfrag*{w}[B][B]{$w$}
 \psfrag*{p}[B][B]{$p$} \psfrag*{d}[B][B]{$1$} \psfrag*{c}[B][B]{$2$} \psfrag*{b}[B][B]{$3$}
\psfrag*{a}[B][B]{6} \begin{center}\includegraphics*[scale=.6]{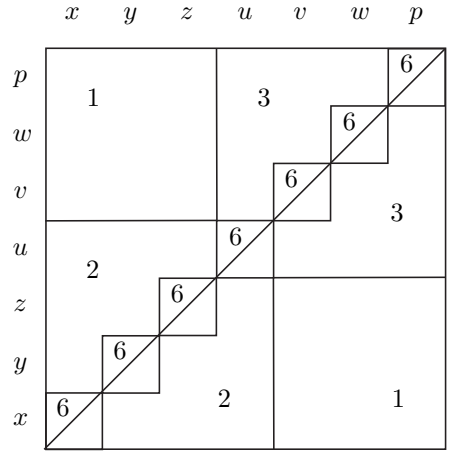}\end{center}
\caption{A graphical example of verifying the index conditions.}\labelp{F:fig1}
\end{figure}

There is a kind of converse to the GCD Theorem that is true.

\begin{theorem}\labelp{T:3.9}Suppose that a group pair $\mc F=\pair G\vp$ consists of finite\comma
cyclic groups\comma with
\[\G{}=\langle \G x:x\in I\rangle\qquad\text{and}\qquad \vp=\langle\vphi\xy:\pair xy\in\mc E\rangle\comma\] and
satisfies the four frame conditions\per  If
\[\map  xy=|\G x/\h\xy|\] for every pair $\pair xy$ in $\mc E$\comma where $\h\xy$ is the
kernel of $\vphi\xy$, then the resulting system
\[m=\langle\map  xy:\pair x y\in\mc E\rangle\] satisfies the index conditions\per
\end{theorem}

\begin{proof}Index condition (i) is satisfied, by the very  definition of $\map  xy$. As regards index condition
(ii), it follows from frame condition (i). In more detail, $\vphi\xx$ is the identity automorphism of
$\G x/\{0\}$, by frame condition (i), so $\h \xx=\{0\}$ and therefore
\[\map  xx=|\G x/\h\xx|=|\G x|\per\]

Frame condition (ii) says that $\vphi\yx=\vphi\xy\mo$.  In particular,
\[\h\yx=\k\xy\qquad\text{and}\qquad\k\yx=\h\xy\comma\]
and therefore
\[\map  yx =|\G y/\h\yx|=|\gr xy|=|\gl xy|=\map   xy\per\]

Turn now to index condition (iv).  Because $\map  xy$ is the index
of $\h\xy$ in $\G x$, the subgroup $\h\xy$ must consist of the
multiples of $\map  xy$ modulo $\cc nx$.  Similarly, $\h\xz$ must
consist of the multiples of $\map  xz$ modulo $\cc nx$.  It follows
from cyclic group theory that the composite group $\h\xy\scir\h\xz$
is generated  by $d=\gcd\pair {\map  xy}{\map  xz}$, and therefore
\begin{equation*}\tag{1}\labelp{Eq:thm6.1}
  d=|\G x/(\h\xy\scir\h\xz)|\per
\end{equation*}
Similar arguments show that  $\k\xy\scir\h\yz$ is generated  by
$d'=\gcd\pair {\map  xy}{\map  yz}$, so that
\begin{equation*}\tag{2}\labelp{Eq:thm6.2}
  d'=|\G y/(\k\xy\scir\h\yz)|\comma
\end{equation*}
and  $\k\xz\scir\k\yz$ is generated  by $d''=\gcd\pair {\map  xz}{\map  yz}$, so that
\begin{equation*}\tag{3}\labelp{Eq:thm6.3}
  d''=|\G z/(\k\xz\scir\k\yz)|\per
\end{equation*}  The induced isomorphism $\hvphs\xy$ maps the quotient
$\G x/(\h\xy\scir\h\xz)$ isomorphically to the quotient
$\G y/(\k\xy\scir\h\yz)$, by frame condition (iii), so \refEq{thm6.1} and
\refEq{thm6.2} imply that $d=d'$.    Similarly, the induced isomorphism $\hvphs\yz$
maps the quotient $\G y/(\k\xy\scir\h\yz)$ isomorphically to  the quotient
$\G z/(\k\xz\scir\k\yz)$, so \refEq{thm6.2} and \refEq{thm6.3} imply that $d'=d''$.
Combine these observations with the definitions of $d$, $d'$, and $d''$ to conclude that index condition (iv) holds.
\end{proof}

\section{Regular elements and indices}\labelp{S:sec4}

\newcommand\ins[1]{\textnormal{index}(#1)}

We now fix a measurable relation algebra $\f A$.  Thus, the identity
element in $\f A$ is the sum of a set $I$ of subidentity atoms, and
each atom $x$ in $I$ is measurable in the sense that the square
$x;1;x$ is a sum of non-zero functions below it.  These functions
are atoms and  form a group  $\G x$ under the operations of relative
multiplication and converse in $\f A$, with $x$ as the group
identity element, by Lemmas 3.2 and 3.3 in \cite{givand3}.  We
assume that each such group is   finite and cyclic.  All elements
are assumed to be in $\f A$.   The left and right stabilizers of an
element $a$ in $\f A$ are the sets
\[\h a=\{f\in \G x:f;a=a\}\qquad\text{and}\qquad \k a=\{g\in \G y: a;g=a\}, \]
and these stabilizers are (normal) subgroups of $\G x$ and $\G y$ respectively.
For measurable atoms $\wx$ and $\wy$, an
 element $a\le \xoy$  is called \textit{regular} if
\[a;a\conv=\suml \L_\al\qquad\text{and}\qquad a\conv;a=\suml \R_\al\per\]
It turns out that regular elements have some of the properties of
atoms. In particular, every atom is regular, by Partition Lemma 4.11
in \cite{givand3}.

\begin{df} \labelp{D:index} For each regular element $a$\comma  define the \emph{index} of $a$ to
be the cardinality of the quotient algebra $\G x/\h a$\comma
in symbols,
\[\ins a=|\G x/\h a|=|\G y/\k a|\per\]
\end{df}

\noindent In other words, the index of $a$ is the number of cosets
that the normal subgroup $\h a$ has in $\G x$, or, equivalently,
that the normal subgroup $\k a$ has in $\G y$.
Moreover, for every coset $H$ of $\h a$ there is a uniquely
determined coset $K$ of $\k a$ such that \[ H;a=a;K\comma\] and
conversely, so that the function $\vphi a$ from $\G x/H_{a}$ to $\G
y/K_{a}$ defined by \[\vphi a(H)=K\qquad\text{if and only if}\qquad
H;a=a;K\] is a bijection, and actually a quotient isomorphism.

\begin{lm}\labelp{L:index1} If $\G y$ is a cyclic group\comma and $a\le x;1;y$
and $b\le y;1;z$ are regular elements\comma then $a;b\le \xoz$ is a regular element, and
\[\ins{a;b}=\gcd\pair{\ins a}{\ins b}\per\]
\end{lm}

\begin{proof} For notational  convenience, assume that $\G y$ is the cyclic
group $\mathbb{Z}_n$ under the operation of addition modulo $n$, and write
\begin{equation*}\tag{1}\labelp{Eq:index1.1}
  k=\ins a=|\G y/\k a|\qquad\text{and}\qquad \ell=\ins b=|\G y/\h b|\per
\end{equation*}
Thus, $\k a$ and $\h b$ consist  of the multiples of $k $ and $\ell$
modulo $n$ respectively. The complex product \[\k a;\h b=\{f;g:f\in
\k a\text{ and } g\in \h b\}\] consists of the multiples of
$\gcd\pair k\ell$ modulo $n$, by cyclic group theory.  Consequently,
\begin{equation*}\tag{2}\labelp{Eq:index1.2}
 |\G y/(\k a;\h b)|=\gcd\pair k\ell=\gcd\pair{\ins a}{\ins b}\comma
\end{equation*} by cyclic group theory and \refEq{index1.1}. According to
Relative Product Theorem 5.14 in \cite{givand3}, $a;b$ is a regular element,
and the isomorphism  $\vphi a$ from $\G x/\h a$ to $\G y/\k a$ induces an
isomorphism $\hvphs a$ from $\G x/\h{a;b}$ to $\G y/(\k a;\h b)$\per  In particular,
\begin{equation*}\tag{3}\labelp{Eq:index1.3}
 |\G y/\h{a;b}|=|\G x/(\k a;\h b)|\per
\end{equation*} By definition, \begin{equation*}\tag{4}\labelp{Eq:index1.4}
   \ins {a;b}=|\G x/\h{a;b}|\per
\end{equation*}
 Combine \refEq{index1.2}--\refEq{index1.4} to arrive at the desired equation.
\end{proof}

\begin{lm}\labelp{L:index2}  Let $x$ and $y$ be measurable atoms\comma
and $a,b\le x;1;y$ regular elements with $a\le b$\per If $\ins a=\ins b$\comma
then $a=b$\per
\end{lm}

\begin{proof} The assumption $a\le b$ implies that
\begin{equation}\tag{1}\labelp{Eq:index2.1}
\h a\seq\h b\comma
\end{equation}
  by Lemma 4.14 in \cite{givand3}. The assumption on the indices implies that
  $\h a$ and $\h b$ have the same number of  cosets in $\G x$. These
  cosets partition $\G x$, so it follows from \refEq{index2.1} that the
  inclusion symbol in \refEq{index2.1} may be replaced by equality. Apply
  (the reverse implication of) Lemma 4.14 from \cite{givand3}   to conclude that $a=b$.
\end{proof}

The next lemma is a known result from cyclic group theory.
\begin{lm}\labelp{L:index3} If $H$ and $K$ are subgroups of $\mathbb Z_n$ with
relatively prime indices $h$ and $k$\comma then $\mathbb Z_n=H\scir K$\comma  and
the cosets of $H\cap K$ in $\mathbb Z_n$ are the sets $(H+i)\cap(K+j)$ for $0\le i<h$
and $0\le j<k$.  Each coset has  $n/(h\cdot k)$ elements\per
\end{lm}

\begin{lm}\labelp{L:index4}  Let $x$ and $y$ be measurable atoms\comma and assume
that $a,b\le x;1;y$ are regular elements with $a\cdot b\neq 0$\per If $\ins a$
and $\ins b$ are relatively prime\comma then \[\ins {a\cdot b}=\ins a\cdot \ins b\per\]
\end{lm}

\begin{proof} Product Theorem 4.13 from \cite{givand3} and the hypothesis that $a\cdot b\neq 0$ imply that
\begin{equation}\tag{1}\labelp{Eq:index3.1}
\h{a\cdot b}=\h a\cap\h b\per
\end{equation}
The cosets of $\h a$ and $\h b$ have the form $\h a + i$ and $\h b + j$ for $0\le i<\ins a$
and $0\le j < \ins b$ respectively.  Apply \refL{index3} (with $\h a$ and $\h b$ in place
of $H$ and $K$ respectively), and use the assumption that the indices of $a$ and $b$ are relatively prime,
to see that
\[\langle(\h a + i)\cap(\h b + j):i<\ins a\text{ and }j < \ins b\rangle\]is a coset
system for \refEq{index3.1} in $\G x$, consisting of $\ins a\cdot\ins b$ cosets, and
each coset has $\ell $ elements, where
\[|\G x|=\cc nx =\ins a\cdot\ins b\cdot\ell\per\]  Consequently, the index of \refEq{index3.1} in $\G x$ is
\begin{equation}\tag{2}\labelp{Eq:index3.2}
|\G x/(\h a\cap\h b)|=\ins a\cdot\ins b\comma
\end{equation} by the definition of the index.  Apply \refEq{index3.1}, \refEq{index3.2},
and the definition of the index of an element to conclude that
\[\ins{a\cdot b}=|\G x/ \h{a\cdot b}|=|\G x/(\h a\cap\h b)|=\ins a\cdot\ins b\per\]
\end{proof}

\begin{cor}\labelp{C:index5}  Let $x$ and $y$ be measurable atoms\per
If $\cc ai\le x;1;y$ for $i=1,\dots,n$ are regular elements with pairwise relatively
prime indices\comma and if $\tprod\limits_{i=1}^n \cc ai\neq 0$, then
\[\ins {\tprod\limits_{i=1}^n \cc ai}=\prod\limits_{i=1}^n \ins {\cc ai}\per\]
\end{cor}The proof is by induction on $n$.  The details are left to the reader.
Observe that the symbol $\tprod$ is being used in two different ways
in the preceding corollary. In its first occurrence, it denotes the
Boolean operation of    multiplication on finite sequences of
elements in a Boolean algebra.  In its second occurrence, it denotes
the arithmetic operation of   multiplication on finite sequences of
natural numbers.  This double usage of the symbol is very common and
should not cause readers any confusion. A similar remark applies to
the use of the symbol $\,\cdot\,$ in \refL{index4}.

\begin{lm}\labelp{L:index6}  Let $x$ and $y$ be measurable atoms\comma and $a,b\le x;1;y$ regular
elements\per If the indices of $a$ and $b$ are relatively prime, then $a\cdot b$ is a regular
element below $x;1;y$\per
\end{lm}

\begin{proof} The assumption on the indices of  $a$ and $b$ implies that the left
stabilizers $\h a$ and $\h b$ have relatively prime indices, and therefore
\begin{equation}\tag{1}\labelp{Eq:index6.1}
\h a;\h b=\h b;\h a=\G x\comma
\end{equation}
by \refL{index3}. The definition of $\h b$ as the left stabilizer of $b$ means that
\begin{equation}\tag{2}\labelp{Eq:index6.2}
\h b; b=  b\per
\end{equation}
Consequently,
\begin{equation}\tag{3}\labelp{Eq:index6.3}
\tsum\h a; b=\tsum\h a;\h b;b=\tsum\G x;b=(x;1;x);b=x;1;y\comma
\end{equation}
by \refEq{index6.2}, \refEq{index6.1}, the fact that $\tsum\G x=x;1;x$, by Corollary 3.5
in \cite{givand3}, and \refL{square}(ii).  Similarly,
\begin{equation}\tag{4}\labelp{Eq:index6.4}
\tsum\h b;a=\tsum\h b;\h a;a=\tsum\G x;a=(x;1;x);a=x;1;y\per
\end{equation}
Combine \refEq{index6.3} and \refEq{index6.4} to arrive at
\begin{equation}\tag{5}\labelp{Eq:index6.5}
\tsum\h a; b=\tsum\h b; a\per
\end{equation}

Use the assumption that $a$ and $b$ are regular elements, together with the
definition of such elements, and apply it to \refEq{index6.5} to obtain
\begin{equation}\tag{6}\labelp{Eq:index6.6}
a;a\conv; b=b;b\conv; a\per
\end{equation}
Part (i) of Product Theorem 4.18 in \cite{givand3} says that the condition in
\refEq{index6.6} is equivalent to the assertion that $a\cdot b\neq 0$.  Apply
Product Theorem 4.13 from \cite{givand3} to conclude that $a\cdot b$ is a regular
element below $x;1;y$.
\end{proof}

\newcommand\gll[1]{\G x/H_{#1}}
\newcommand\grr[1]{\G y/K_{#1}}
\newcommand\gldhh[3]{\G {#1}/(H_{#2};H_{#3})}
\newcommand\gldkh[3]{\G {#1}/(K_{#2};H_{#3})}
\newcommand\gldkk[3]{\G {#1}/(K_{#2};K_{#3})}

\begin{cor}\labelp{C:index7}  Let $x$ and $y$ be measurable
atoms\per If $\cc ai\le x;1;y$ for $i=1,\dots,n$ are regular
elements with pairwise relatively prime indices\comma
then $\tprod\limits_{i=1}^n \cc ai$ is a regular element below $x;1;y$.
\end{cor}The proof is by induction on $n$.  The details are left to the reader.

\section{The representation theorem}\labelp{S:sec5}

We continue with the assumption that $\f A$ is a measurable relation algebra.
The assumption that the groups are finite and cyclic implies that the algebra $\f A$
is finitely measurable, and hence automatically atomic, by Theorem 8.3 in \cite{givand3}.
The goal is to show that $\f A$ is essentially isomorphic to a group relation
algebra $\cra GF$ for  one of the group frames $\mc F$ constructed in  GCD
\refT{gcd}.  This means that we must show that the completion of $\f A$ (in the
sense of the minimal complete extension of $\f A$) is isomorphic to $\cra GF$.
Scaffold Representation Theorem 7.4 in \cite{givand3} says that an atomic measurable
relation algebra $\f A$ is essentially isomorphic to some group relation algebra if
and only if it has a scaffold.  A \emph{scaffold} in $\f A$ is a system
$\langle\cc a\xy:\pair xy\in \mc E\rangle$ of atoms in $\f A$ that satisfies the
following conditions for all $\pair xy$ and $\pair yz$ in $\mc E$.
\begin{enumerate}
\item[(i)] $\cc a \xx=x$\per
\item[(ii)] $\cc a\yx=\cc a\xy\conv$\per
\item [(iii)] $\cc a\xz\le\cc a \xy;\cc a \yz$\per
\end{enumerate}
Thus, to prove the desired representation theorem for $\f A$, it
suffices to construct a scaffold in $\f A$.

Fix measurable atoms $x$ and $y$ in $\f A$ with $\pair x y$ in $\mc
E$, and consider a regular element $a\le \xoy$. Regular elements are
always non-zero, by Lemma 4.4 in \cite{givand3}, so $a\neq 0$.  If
functions $f$ and $g$ belong  to the same coset $H$ of $\h a$, then
the left translations $f;a$ and $g;a$ of $a$ are equal, and if they
belong to different cosets of $\h a$, then $(f;a)\cdot(g;a)=0$, by
Lemma 4.6(iii) in \cite{givand3}. Thus, it makes sense to write
$H;a$ whenever $H$ is a coset of $\h a$, and this just denotes the
element $f;a$ for some (any) element $f$ in $H$.  The left
translations $b=H;a$ by cosets $H$ of $\h a$ are mutually disjoint,
regular elements below $\xoy$ with the same normal stabilizer $\h a$
as $a$, and in fact these left translations form a partition of
$\xoy$, by Partition Lemma 4.9 in \cite{givand3}.  Similar remarks
apply to the right translations $a;K$ of $a$ by cosets $K$ of $\k a$
(in $\G y$).

Since $\f A$ is atomic, there must be an atom below $\xoy$, and such
atoms are regular elements with the same stabilizer, by Partition
Lemma 4.11 in \cite{givand3}.  Write $\h \xy$ and $\k\xy$ for the
left and right stabilizer of these atoms, and for a fixed atom $a$,
write $\vphi\xy$ for the quotient isomorphism $\vphi a$.  The choice
of $\vphi\xy$ is dependent on $a$, but a system of atoms can be
chosen with the following properties (see, for example, the remarks
in Section 7 of \cite{givand3}).
\begin{enumerate}
\item[(P1)] $\h\xx=\{x\}$, and $\vphi\xx$ is the identity isomorphism of $\gll \xx$\per
\item[(P2)] $\h \yx=\k\xy$, \qquad $\k\yx =\h \xy$,\qquad  and\qquad $\vphi\yz=\vphi\xy \mo$\per
\item[(P3)] $\vphi\xy[\h\xy;\h\xz]=\k\xy;\h\yz$\comma \qquad $ \vphi\yz[\k\xy;\h\yz]=\k\xz;\k\yz$\comma \qquad and\qquad
$\\ \vphi\xz[\h\xy;\h\xz]=\k\xz;\k\yz$\per
\item[(P4)] $ |\gldhh x\xy\xz|=|\gldkh y\xy\yz|=|\gldkk z\xz\yz|\per$
\end{enumerate}
In fact, the  isomorphism $\vphi\xy$ induces an isomorphism
$\hvphs\xy$ from $\gldhh x \xy\xz$ to $\gldkh y\xy\yz$, while
$\vphi\yz$ induces an isomorphism $\hvphs\yz$ from $\gldkh y\xy\yz$
to $\gldkk z\xz\yz$.  Property (P4) is an immediate consequence of
this observation.

Put $\map xy =|\gll \xy|$.

\begin{lm}[Index Lemma]\labelp{L:index70} Suppose $\pair xy$ and $\pair yz$ are
pairs of measurable atoms in $\mc E$\per
\begin{enumerate}
  \item[(i)] $\map xx=|\gll\xx|=|\G x/\{x\}|=|\G x|$\per
  \item[(ii)] $\map yx=\map xy$\per
  \item[(iii)] $\gcd\pair{\map xy}{\map xz}=\gcd\pair{\map xy}{\map  yz}=\gcd\pair{\map xz}{\map yz}$\per
\end{enumerate}
\end{lm}

\begin{proof} Property (P1) and the definition of $\map  xx $ imply
that\[ \map xx=|\gll\xx|=|\G x/\{0\}|=|\G x|\per\]
 Property (P2) and the definitions of $\map  xy $ and
 $\map yx$ imply that\[ \map yx=|\G y/\h\yx|=|\grr \xy|=|\gll\xy|=\map xy\per\]

Since $\h\xy$ and $\h\xz$ have indices $\map xy$ and $\map xz$, they
are respectively generated by (copies of) the integers $\map xy$ and
$\map xz$ modulo $\cc nx$.  Consequently, the group composition
$\h\xy;\h\xz$ (under the operation of relative multiplication in $\f
A$) is generated by the element $\gcd\pair{\map xy}{\map xz}$, so
that
\begin{equation}\tag{1}\labelp{Eq:index7.1}
 \gcd\pair{\map xy}{\map xz}=|\gldhh x \xy\xz|\per
\end{equation} Similarly,
\begin{equation}\tag{2}\labelp{Eq:index7.2}
 \gcd\pair{\map xy}{\map yz}=|\gldkh y \xy\yz|\per
\end{equation}
\begin{equation}\tag{3}\labelp{Eq:index7.3}
 \gcd\pair{\map xz}{\map yz}=|\gldkk z \xz\yz|\per
\end{equation} Combine \refEq{index7.1}--\refEq{index7.3} with property (P4) to arrive at (iii).
\end{proof}

\newcommand\sik{\sim_k}
\newcommand\akxy[2]{a_{#1}^{#2}}

Fix a prime number $p$ for the next definition and two lemmas, and in terms of this prime,
define a binary relation $
\sik$ as follows.
\begin{df}\labelp{D:index8}
For $x$ and $y$ in $I$, define $x\sik y$ if and only if $x= y$\comma  or $\pair xy$
is in $\mc E$ and $p^k$ divides $\map xy$.
\end{df}

\begin{lm}\labelp{L:index9} The relation $\sik$ is an equivalence relation
on the set $I$\per
\end{lm}
\begin{proof} The relation is automatically reflexive, by its very definition. For
symmetry, suppose that $x\sik y$ and $x\neq y$.  In this case $\pair xy$ belongs
to $\mc E$ and $p^k$ divides $\map xy$.  The relation $\mc E$ is symmetric, so it
contains $\pair yx$, and $\map yx=\map xy$, so $p^k$ divides $\map yx$. Thus,
$y\sik x$, by \refD{index8}.

Turn now to transitivity.  Assume that $x\sik y$ and $y\sik z$.  If
two of these atoms are equal, then the proof of transitivity is
trivial, so suppose that all three atoms are distinct.  The
hypotheses and \refD{index8} imply that $p^k$ divides $\map xy$ and
$\map yz$, so it divides their  greatest common divisor.  Since
\[ \gcd\pair{\map xz}{\map yz}=\gcd\pair{\map xy}{\map yz}\]
 by Index \refL{index70}, it follows that $p^k$ divides $\map xz$.  Also,
 $\mc E$ is transitive, so the pair $\pair xz$ is  in $\mc E$.
 Therefore, $x\sik z$, by \refD{index8}.
\end{proof}

\begin{lm}\labelp{L:index10} For each prime $p$\comma there is a system
of elements \[\langle \akxy \xy k:k\ge 0\text{ and }x\sik y\rangle\] with
the following properties  whenever $x\sik y$ and $y\sik z$.
\begin{enumerate}
  \item[(i)] $ \akxy \xy k$ is a regular element below $\xoy$\comma and $\akxy\xy k=x;1;y$ when $x\neq y$ and  $k=0$.
   \item[(ii)]  If $x=y$, then $\akxy \xy k = x$\per
  \item[(iii)]  If $x\neq y$, then $\ins{\akxy \xy k} = p^k$\per
  \item[(iv)] $\akxy \yx k=(\akxy \xy k)\conv$\per
  \item[(v)] $\akxy \xz k\le \akxy \xy k;\akxy \yz k$\comma and equality holds when $x\neq z$\per
  \item[(vi)] $\akxy \xy k\le \akxy \xy {k-1}$ for $k\ge 1$\per
\end{enumerate}
\end{lm}

\begin{proof}  The construction is by induction on $k$ starting at $k=0$.
In this case, the definition of $\akxy \xy k$ is dictated by the first two conditions:
\begin{equation*}\tag{1}\labelp{Eq:index10.1}
\akxy \xy 0 =\begin{cases}   x &\text{if $x = y$,} \\
\xoy &\text{if $x\neq y$.}
\end{cases}
\end{equation*}
It is not difficult to verify that in this case properties (i)--(vi)
hold. If $x=y$, then the element $x\le\xox$ is regular with left and
right stabilizers $\{x\}$.  If $x\neq y$, then $\xoy$ is regular
with  left and right stabilizers $\G x$ and $\G y$ respectively.
Consequently,
\[\ins{\akxy \xy k}=|\G x/
\G x| = 1 = p^0\comma\] so properties (i)--(iii) hold.

For property (iv), observe that if $x=y$, then
\[\akxy \yx 0=y=x=x\conv=(\akxy \xy k)\conv,\] by \refL{laws}(v).
On the other hand, if $x\neq y$, then\[\akxy \yx  0=\yox=(\xoy)\conv=(\akxy \xy 0)\conv\comma\]
by \refL{square}(i).

The verification of property (v) breaks down into cases.  If $x=y$, then
\[\akxy \xy k;\akxy\yz k=x;\akxy \yz k=x;\akxy \xz k= \akxy \xz k\comma
\]
by condition (ii), the assumption that $x=y$, and \refL{square}(iii).  A similar argument
applies if $y=z$.  If $x=z$, then
\[\akxy \xy k;\akxy\yz k=\akxy \xy k;\akxy\yx k=\akxy\xy k;(\akxy \xy k)\conv=\tsum\h\xy\ge   x=\akxy \xx k=\akxy\xz k\comma\]
by the assumption that $x=z$, condition (iv) (which has already been shown to hold when $k=0$),
the regularity of $\akxy \xy k$ from condition (i) (which has already been shown to hold when $k=0$),
the fact that $x$ is in its left stabilizer $\h\xy$, condition (ii) (which has already been
shown to hold when $k=0$), and the assumption $x=z$.  If $x$, $y$, and $z$ are pairwise distinct, then
\[\akxy \xy k;\akxy\yz k=(\xoy);(\yoz)=\xoz=\akxy \xz k\comma
\]by the definition in \refEq{index10.1} and \refL{square}(ii).

Condition (vi) holds vacuously.

\newcommand\sil[3]{{#1}\sim_{#2}{#3}}
\newcommand\br[1]{\bar{#1}}

Assume now that $\akxy \xy{k-1}$ has been defined for all pairs $\pair xy$
with $\sil x{k-1} y$ so that conditions (i)--(vi) hold (with $k-1$ in place of $k$, and $k\ge 1$).
For each measurable atom $x$ in $I$, choose a representative $\br x$ of the equivalence class $x/\sik \,\,$.
Thus,
\begin{equation}\tag{2}\labelp{Eq:index10.2}
x \sik y\qquad\text{if and only if}\qquad \br x=\br y\per
\end{equation}

\newcommand\cb[2]{c_{{#1}{\br {#2}}}}
\newcommand\cbb[2]{c_{{\br {#1}}{#2}}}

The next step is to construct for each element $y\sik x$  an element
$\cb yx\le y;1;\br x$ as follows.  If $y=\br x$, put
\begin{equation}\tag{3}\labelp{Eq:index10.3}
\cb yx=\br x\per
\end{equation}
If $y\neq \br x$, then write $b= \akxy {y\br x}{k-1}$, and fix an
atom $d\le b$. Such an atom exists because $b$ is a regular element,
by the induction hypothesis for  condition (i), and hence $b$ is
non-zero.  The  atomicity of $\f A$ implies that every non-zero
element is above an atom. The element $b$ has index $p^{k-1}$ by the
induction hypothesis for condition (iii), so the subgroup $\h b$ is
generated by $p^{k-1}$.  Let $L$ be the subgroup of $\G x$ generated
by $p^k$. The element $p^{k-1}$ generates $\h b$, and $p^k$ is a
multiple of $p^{k-1}$, so $p^k$ belongs to $\h b$, and therefore $L$
is included in $\h b$. On the other hand, $\map y{\br x}$ generates
$\h {y\br x}$, and $p^k$ divides $\map y{\br x}$, so $\h {y{\br x}}$
must be included in $L$.  Thus,
\begin{equation}\tag{4}\labelp{Eq:index10.4}
\h{y\br x}\seq L\seq \h b.
\end{equation}  Observe that
\begin{equation}\tag{5}\labelp{Eq:index10.5}
|\G x/L|=p^k\comma
\end{equation}
since $p^k$ generates $L$.

Put
\begin{equation}\tag{6}\labelp{Eq:index10.6}
\cb yx=\tsum L;d\per
\end{equation}
Because $L$ is a subgroup of $\G x$, and $d$ is an atom satisfying $d\le b \le y;1;\br x$
(the second inequality uses the induction hypothesis for condition  (i)),
the element $\cb yx$ is a regular element below $y;1;\br x$, and its left
stabilizer is $L$, by Theorem 9.1 in \cite{givand3}.  Consequently,
 \begin{equation}\tag{7}\labelp{Eq:index10.7}
\ins{\cb yx}=|\G x/L|=p^k\comma
\end{equation} by \refEq{index10.5}.  The preceding observations and
\refEq{index10.3} show that conditions (i)--(iii) of the lemma hold
with $\cb yx$ in place of $\akxy {y\br x} k$.

For $y\neq \br x$. define
\begin{equation}\tag{8}\labelp{Eq:index10.8}
\cbb xy=\cb yx\conv\comma
\end{equation}
and observe that since the converse of a regular element below $y;1;\br x$ is a
regular element below $\br x ;1;y$, with the left and right stabilizers reversed,
by Converse Theorem 5.13 from \cite{givand3}, the element defined in \refEq{index10.8}
must be regular, below $\br x;1;y$, and   have the same index as $\cb yx$, so that
 \begin{equation}\tag{9}\labelp{Eq:index10.9}
\ins{\cbb xy}=p^k\comma
\end{equation}
by \refEq{index10.7}.

For arbitrary $x$ and $y$ in $I$ with $x\sik y$, define
\begin{equation*}\tag{10}\labelp{Eq:index10.10}
\akxy \xy k =\begin{cases}   x &\text{if $x = y$,} \\
\cb xx;\cbb xy &\text{if $x\neq y$.}
\end{cases}
\end{equation*}
Observe that $\akxy \xy k$ is well defined by \refEq{index10.2}. The
relative product of a regular element below $x;1;\br x$ and a
regular element below $\br x;1;y$ is a regular element below $\xoy$,
by Relative Product Theorem 5.16 in \cite{givand3}, so condition (i)
of the lemma is satisfied when $x\neq y$, and it is trivially
satisfied when $x=y$.  Condition (ii) is automatically satisfied, by
\refEq{index10.10}.

Turn to the verification of condition (iii).  Assume $x\sik y$ are
distinct. Use   \refEq{index10.10}  and \refL{index1} (with $\cb xx$
and $\cbb xy$ in place of $a$ and $b$ respectively)
 to obtain
 \begin{equation}\tag{11}\labelp{Eq:index10.11}
\ins{\akxy \xy k}=\ins{\cb xx;\cbb xy}=\gcd\pair{\ins{\cb xx}}{\ins{\cbb xy}}\per
\end{equation}
The value of $\ins{\cb xx}$ is either $p^k$ or $|\G x|$ according to
whether $x\neq \br x$ or $x=\br x$, by \refEq{index10.7} and
\refEq{index10.3}, and similarly for $\cbb xy$, by
\refEq{index10.7},\refEq{index10.3}, and \refEq{index10.8}.  At
least one of them must be $p^k$ since $y\neq x$, so the value of
\refEq{index10.11} is $p^k$. This completes the verification of
condition (iii).

Take up now condition (iv).  If $x=y$, then
\[\akxy\yx k=y=x=x\conv=(\akxy \xy k)\conv\comma\]

by \refEq{index10.10} and \refL{laws}(v).  If $x\sik y$ are
distinct, then\begin{equation*} \akxy\yx k=(\cb y y;\cbb y x)=\cbb
yy\conv;\cb xy\conv= (\cb xy;\cbb yy)\conv=(\cb xx;\cbb
xy)\conv=(\akxy\xy k)\conv\comma\end{equation*} by
\refEq{index10.10}, \refEq{index10.8},   the second involution law,
\refEq{index10.2}, and \refEq{index10.10}.

The next verification is of condition (v). Assume that $x\sik y\sik
z$, and consider first the cases when at least two of the three
atoms are equal.  If $x=y$, then\[\akxy\xy k;\akxy \yz k=x;\akxy \yz
k=x;\akxy \xz k= \akxy \xz k\comma\] by \refEq{index10.10}, the
assumption that $x=y$, and \refL{square}(iii).  The argument when
$y=z$ is similar.    If $x=z$, then
\[\akxy\xy k;\akxy \yz k=\akxy\xy k;\akxy \yx k=\akxy\xy k;(\akxy \xy k)\conv=\tsum\h\xy\ge x=\akxy\xz k\comma\]
by the assumption that $x=z$, condition (iv), the regularity of $\akxy \xy k$,
which is ensured by condition (i), the fact that $\h\xy$ contains $x$, monotony,
and \refEq{index10.10}.

Assume now that the atoms $x$, $y$, and $z$ are mutually distinct.
The element $\cbb yy$ is regular, and it is the converse of $\cb
yy$, by \refEq{index10.8}, so
\begin{equation}\tag{12}\labelp{Eq:index10.12}
\cbb yy;\cb yy=\cb yy\conv;\cb yy=\tsum\k {y\br y}\ge \br y
\end{equation}
(the last step uses monotony and the fact that $\br y$ is in
$\k{y\br  y}$). Consequently, \begin{multline*} \akxy\xy k;\akxy\yz
k=(\cb xx;\cbb xy);(\cb yy;\cbb y z)=(\cb xx;\cbb yy);(\cb yy;\cbb y
z) \\ =\cb xx;(\cbb yy;\cb yy);\cbb y z \ge  \cb xx;\br y;\cbb y z=
\cb xx; \cbb y z= \cb xx; \cbb x z=\akxy\xz k\comma \end{multline*}
by \refEq{index10.10}, \refEq{index10.2}, the associative law,
\refEq{index10.12} and monotony,  \refL{square}(iii),
\refEq{index10.2}, and \refEq{index10.10}.  This argument shows that
\begin{equation}\tag{13}\labelp{Eq:index10.13}
\akxy\xz k\le \akxy\xy k;\akxy \yz k\per
\end{equation}
On the other hand,
\begin{multline}\tag{14}\labelp{Eq:index10.14}
\ins{ \akxy\xy k;\akxy \yz k}=\gcd\pair{\ins{\akxy \xy k}}{\ins{\akxy \yz k}}\\=\gcd\pair{ p^k}{ p^k}=p^k=\ins{\akxy\xz k}\comma
\end{multline} by condition (i), \refL{index1}, condition (iii), and the assumption
on $x$, $y$, and $z$. Use condition (i), \refEq{index10.13}, \refEq{index10.14}, and
\refL{index2} (with $\akxy\xz k$ and $\akxy\xy k;\akxy\yz k$ in place of $a$ and $b$) to conclude that
\begin{equation*}
\akxy\xz k= \akxy\xy k;\akxy \yz k\per
\end{equation*}

Turn finally to the verification of condition (vi).  If $x=y$, then
\[\akxy\xy k=x=\akxy\xy{k-1}\comma\] by \refEq{index10.10} and the induction
hypothesis for condition (ii), so condition (vi) holds in this case.  Assume now that $x\sik y$ are
distinct.  At least one of $x$ and $y$ must be different from $\br x$, say it is $y$.  If $x=\br x$, then
\begin{equation}\tag{15}\labelp{Eq:index10.15}
\akxy\xy k=\cb x x;\cbb x y=\br x;\cbb x y=\cbb x y=\cb
yx\conv\comma
\end{equation}
by \refEq{index10.10}, \refEq{index10.3}, \refL{square}(iii), and
\refEq{index10.8}.  Write $b=\akxy {y\br x}{k-1}$.  The element $\cb
yx$ is defined to be $\tsum L;d$, where $d$ is some atom below $b$,
and $L$ satisfies the inclusions in \refEq{index10.4}.  The second
inclusion in \refEq{index10.4} implies that $L;d\seq \h b;d$, and
therefore
\begin{equation}\tag{16}\labelp{Eq:index10.16}
\cb y x=\tsum L;d\le \tsum \h b;d\le \tsum \h b ;b=b=\akxy {y\br x} {k-1}\comma
\end{equation}
by \refEq{index10.6}, monotony, the fact that $d\le b$, the definition of $\h b$ as the
stabilizer of $b$, and the definition of $b$.  Apply \refEq{index10.16}, monotony, and the
induction hypothesis for condition (iv) to arrive at
\begin{equation}\tag{17}\labelp{Eq:index10.17}
\cb y x\conv\le(\akxy {y\br x} {k-1})\conv=\akxy {\br x y}{k-1}\per
\end{equation}  With the help of \refEq{index10.15}, \refEq{index10.17},
\refL{square}(iii), the assumption  that $x=\br x$,    the induction
hypotheses for (ii) and (v),   and the assumption that $y\neq \br
x$, conclude that\[\akxy \xy k=\cb y x\conv\le \akxy {\br x
y}{k-1}=\br x;\akxy{\br x y}{k-1}= x;\akxy{\br x y}{k-1}=\akxy {x\br
x} {k-1};\akxy {\br x y}{k-1}=\akxy{x y}{k-1}\per\]

Consider finally the case when both $x$ and $y$ are different from
$\br x$. An argument analogous to that of  \refEq{index10.16} shows
that
\begin{equation}\tag{18}\labelp{Eq:index10.18}
\cb x x\le \akxy {x\br x} {k-1}\per
\end{equation}
Also, it follows from \refEq{index10.17} and  \refEq{index10.8}  that
\begin{equation}\tag{19}\labelp{Eq:index10.19}
\cbb xy \le \akxy {\br x y} {k-1}\per
\end{equation}
Compute:\[\akxy \xy k=\cb xx;\cbb xy\le\akxy {x\br x}{k-1};\akxy{\br
x y}{k-1}=\akxy \xy {k-1}\comma\] by \refEq{index10.10},
\refEq{index10.18}, \refEq{index10.19}, monotony, and the induction
hypothesis for (v) (with $\br x$ and $y$ in place of $y$ and $z$
respectively). This completes the proof of the lemma.
\end{proof}

\newcommand\akixy[2]{a_{#1}^{k_{#2}}}
\newcommand\axy[1]{a_{#1}}

For each pair $\pair xy$ in $\mc E$ and each natural number $k\ge 0$
such that $x\sik y$, an element $\akxy\xy k$ has been constructed in
\refL{index10} such that the system of these elements possesses
certain properties.  The next step is to use these elements and
properties in order to construct a scaffold.  Fix a pair $\pair xy$
in $\mc E$, and let
\[\map xy=p_1^{k_1}\cdot\ldots\cdot p_n^{k_n}\] be the decomposition of $\map xy$ into distinct primes.
If $\map xy=1$, write $\map xy=2^0$ for the prime decomposition.
Thus, for  each index $i=1,\dots, n$, we have $x\sim_{k_i}  y$, and
$p_i^{k_i}$ is the largest power of $p_i$ that divides $\map xy$
when $x\neq y$.

We now change notation a bit by letting the prime numbers $\cs pi$
vary as $i$ varies over $1,\dots, n$, and writing $\akixy \xy i$ to
denote the element constructed in \refL{index10} using the prime
$p=\cs pi$ and the natural number $k=\cs ki$. Thus, in contrast to
the lemma, we let $\cs ki$ and $\cs kj$ represent powers of
different primes, namely $\cs pi$ and $\cs p j$, in the notations
$\akixy\xy i$ and $\akixy\xy j$.

\begin{df} For each pair $\pair xy$ in $\mc E$, write
\[\axy \xy=\akixy \xy 1\cdot\ldots\cdot \akixy\xy n=\tprod_{i=1}^n \akixy\xy i\comma\]
where $\map xy=p_1^{k_1}\cdot\ldots\cdot p_n^{k_n}$ is the prime
decomposition of $\map xy$\per\end{df}

\begin{figure}[htb]  \psfrag*{a}[B][B]{$10,800$} \psfrag*{b}[B][B]{$16$} \psfrag*{c}[B][B]{$27$}
\psfrag*{d}[B][B]{$25$}
 \psfrag*{e}[B][B]{$8$}  \psfrag*{f}[B][B]{$9$}  \psfrag*{g}[B][B]{$5$} \psfrag*{h}[B][B]{$4$}
 \psfrag*{i}[B][B]{$3$}  \psfrag*{j}[B][B]{$2$}
 \begin{center}\includegraphics*[scale=.9]{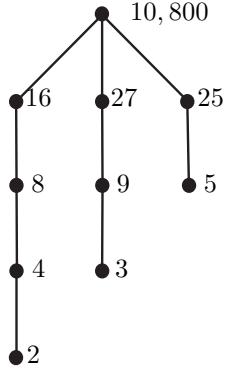}\end{center}
\caption{Diagram of the levels of construction of the scaffold.}\labelp{F:fig2}
\end{figure}

\begin{theorem}[Scaffold Theorem]\labelp{T:scaffold}
The system $\langle\axy\xy:\pair xy\in\mc E\rangle$ is a scaffold in~$\f A$\per
\end{theorem}

\begin{proof}
The first task is to prove that
\begin{equation}\tag{1}\labelp{Eq:index12.1}
\ins{\axy\xy}=\map xy
\end{equation}
for all $\pair xy\in \mc E$\per Consider the case when $x\neq y$.
The element $\akixy\xy i$ is  regular  and below  $\xoy$, and
\begin{equation}\tag{2}\labelp{Eq:index12.2}
\ins{\akixy\xy i}=p^{\cs ki}\comma
\end{equation}
by \refL{index10}(i),(iii).  In particular, these indices are
relatively prime to one another for distinct indices $i$.  There are
now two subcases to consider. If $\map xy > 1$, then
\begin{equation}\tag{3}\labelp{Eq:index12.3}
\axy \xy=\tprod_{i=1}^n \akixy\xy i
\end{equation} is a product of regular elements below $\xoy$ with mutually relatively
prime indices,
so it is a regular element below $\xoy$, and
\begin{equation}\tag{4}\labelp{Eq:index12.4}
\ins{\axy\xy}=\tprod_{i=1}^n\ins{\akixy\xy i}=\tprod_{i=1}^n p^{\cs ki}=\map xy\comma
\end{equation}
by \refC{index5} and \refEq{index12.2}. On the other hand, if $\map xy = 1$, then
\begin{equation}\tag{5}\labelp{Eq:index12.5}
\axy \xy=\xoy\comma
\end{equation}  by \refL{index10}(i), so the left stabilizer of $\axy\xy$ is $\G x$, and consequently
\begin{equation}\tag{6}\labelp{Eq:index12.6}
\ins{\axy\xy}=|
\G x/\G x|=1=\map xy\per
\end{equation}

In the remaining case, $x=y$, so that $\akixy\xy i=x$ for each $i$, by
\refL{index10}(ii), and therefore
\begin{equation}\tag{7}\labelp{Eq:index12.7}
\axy\xy=\axy \xx=x\comma
\end{equation}
by   \refEq{index12.3}.  The left stabilizer of $x$ is $\{x\}$, by the group
identity law, so
\begin{equation}\tag{8}\labelp{Eq:index12.8}
\ins{\axy\xx}=\ins{x}=|\G x/\{x\}|=|\G x|=\map xx\comma
\end{equation}
by \refEq{index12.7}, \refD{index}, and \refL{index70}(i). Combine
\refEq{index12.4}, \refEq{index12.6}, and \refEq{index12.8} to
conclude that \refEq{index12.1} holds in all cases.

The next task is to prove that the element $\axy\xy$ is an atom. As
was pointed out earlier, the algebra $\f A$ is finitely measurable,
and hence atomic.  Since $\axy\xy$ is regular, it is non-zero, and
therefore there must be an atom $d\le \axy\xy$.  The atom $d$ is
regular, by Corollary 4.7 in \cite{givand3}, and its left stabilizer
is $\h\xy$, by Corollary 4.16 in \cite{givand3}. Use the definition
of the index of an element, the definition of $\map xy$, and
\refEq{index12.1} to arrive at
\[\ins d=|\G x/\h\xy|=\map xy=\ins{\axy\xy}\per\]
Apply \refL{index2} to conclude that $d=\axy\xy$, and hence that
$\axy\xy$ is an atom with left stabilizer $\h\xy$.

It remains to verify the three scaffold conditions.  The first one
holds by \refEq{index12.7}.  To verify the second scaffold
condition, it suffices to check the case when $x\neq y$ (see Theorem
4.4 in \cite{giv1}).  If $\map xy>1$, then \refL{index10}(iv)
ensures that
 \begin{equation}\tag{9}\labelp{Eq:index12.9}
\akixy\yx i=(\akixy\xy i)\conv\comma
\end{equation} for each $i$, and consequently,
\begin{equation}\tag{10}\labelp{Eq:index12.10}
\axy\xy\conv=(\tprod_i\akixy\xy i)\conv=\tprod_i(\akixy\xy i)\conv=
\tprod_i\akixy\yx i =\axy\yx\comma
\end{equation}
by \refEq{index12.3} and the assumption that $\map xy>1$,
\refL{laws}(ii),  \refEq{index12.9}, and \refEq{index12.3} (with $y$
and $z$ in place of $x$ and $y$ respectively). If $\map xy =1$, then
$\map yx = 1$, and
\begin{equation}\tag{11}\labelp{Eq:index12.11}
\axy\xy\conv=(\xoy)\conv=\yox =\axy\yx\comma
\end{equation} by \refEq{index12.5} and \refL{square}(i).
Thus, the second scaffold condition holds in all cases, by
\refEq{index12.10} and \refEq{index12.11}.

It remains to verify the third scaffold condition.  Consider pairs
$\pair xy$ and $\pair yz$ in $\mc E$.  If $x=y$, then
\begin{equation}\tag{12}\labelp{Eq:index12.12}
\axy\xz=x;\axy\xz=\axy\xx;\axy\xz=\axy\xy;\axy\yz\comma
\end{equation}
by \refL{square}(iii), \refEq{index12.7}, and the assumption   $x=y$.  A similar argument applies if $y=z$.  If $x=z$, then
\begin{equation}\tag{13}\labelp{Eq:index12.13}
\axy\xz= \axy\xx=x\le
\tsum\h\xy=\axy\xy;\axy\xy\conv=\axy\xy;\axy\yx=\axy\xy;\axy\yz\comma
\end{equation}
by the assumption  $x=z$, \refEq{index12.7}, the fact that $x$ belongs to the
left stabilizer $\h\xy$, monotony, the regularity of the atom $\axy\xy$, and scaffold condition (ii), which has already been shown to hold.

 Assume now that $x$, $y$, and $z$ are all distinct.
 If $\map xy=1$, then \refEq{index12.5} holds, and therefore
\begin{equation}\tag{14}\labelp{Eq:index12.14}
\axy\xy;\axy\yz=\xoy;\axy\yz=\xoz \ge \axy\xz\comma
\end{equation}
by \refEq{index12.5}, \refL{square}(ii), and \refL{index10}(i). A
similar argument applies if $\map yz=1$.  We may therefore assume
that $\map xy>1$ and $\map yz>1$. Thus, $\map xy$ and $\map yz$ each
have at least one prime in their prime decompositions. Suppose
\begin{equation}\tag{15}\labelp{Eq:index12.15}
\map xy=p_1^{k_1}\cdot\ldots\cdot p_n^{k_n}=\tprod_{i=1}^n p_i^{\cs ki}\comma
\end{equation}
so that $\axy\xy$ has the form \refEq{index12.3}.  Forming the
product of $\axy\xy$ with the  element $x;1;y$ does not change the
value of $\axy\xy$, since $\axy\xy$ is below this element, by
\refL{index10}(i).  This amounts to forming the product of $\axy\xy$
with elements of the form
\newcommand\abxy[3]{a_{#1}^{{#2}_{#3}}}
$\abxy \xy\ell j$ in which $\cs \ell j=0$\comma by
\refL{index10}(i). The same reasoning applies to the values of
$\axy\yz$ and $\axy\xz$, so by multiplying such zero powers of
primes into the factorizations of $\map xy$, $\map yz$, and $\map
xz$, we may assume that
\begin{equation}\tag{16}\labelp{Eq:index12.16}
\map yz=p_1^{\ell_1}\cdot\ldots\cdot p_n^{\ell_n}=\tprod_{i=1}^n p_i^{\cs \ell i}\qquad\text{and}\qquad
\map xz=p_1^{j_1}\cdot\ldots\cdot p_n^{j_n}=\tprod_{i=1}^n p_i^{\cs ji}\comma
\end{equation}that is to say, the same primes $\cs pi$ occur in all three factorizations,
some of them raised to the zeroth power. Consequently,
\begin{equation}\tag{17}\labelp{Eq:index12.17}
\axy \yz=\tprod_{i=1}^n \abxy\yz \ell i\qquad\text{and}\qquad \axy
\xz=\tprod_{i=1}^n \abxy\xz ji\per
\end{equation}

Write
\begin{equation}\tag{18}\labelp{Eq:index12.18}
\cs si=\min\{\cs ki,\cs \ell i\}
\end{equation}
for each $i=1,\dots,n$, and observe that
\begin{equation}\tag{19}\labelp{Eq:index12.19}
\gcd\pair{\map xy}{\map yz}=\gcd\pair{\tprod_i p_i^{\cs k i}}{\tprod_i p_i^{\cs \ell i}}=\tprod_i p_i^{\cs s i}\per
\end{equation}
Since
\[\gcd\pair{\map xy}{\map yz}=\gcd\pair{\map xy}{\map xz}\comma\]
by \refL{index70}(iii), it follows from \refEq{index12.19} that
\begin{equation}\tag{20}\labelp{Eq:index12.20}
p_i^{\cs si}\quad\text{divides}\quad \map xz\per
\end{equation}
The definition of $\axy\xy$ implies that $\cs ki$ is the largest natural number such that
\newcommand\sikk[1]{\sim_{{#1}_i}}
$x\sikk k y$.  Similarly, $\cs \ell i$ is the largest natural number such that $y\sikk\ell z$\per  It follows from \refEq{index12.18} and the definition of the relation $\,\sikk s$ that $x\sikk s y$ and $y\sikk s z$, so
\begin{equation}\tag{21}\labelp{Eq:index12.21}
x\sikk s z\comma
\end{equation} by transitivity. Since $\cs ji$ is the largest natural number such that $x\sikk j z$, it follows from \refEq{index12.20} and \refEq{index12.21} that $\cs si\le\cs ji$, and therefore
\begin{equation}\tag{22}\labelp{Eq:index12.22}
\abxy\xz j i\le\abxy\xz s i\comma
\end{equation}
by \refL{index10}(vi).  Use \refEq{index12.17} and \refEq{index12.22} to conclude that
\begin{equation}\tag{23}\labelp{Eq:index12.23}
\axy\xz=\tprod_i\abxy\xz j i\le\tprod_i \abxy\xz s i\per
\end{equation}
The element $\axy\xz$ is an atom, so in particular, \refEq{index12.23} implies that
\begin{equation}\tag{24}\labelp{Eq:index12.24}
 \abxy\xz s i\neq 0\per
\end{equation}  Use \refC{index5}, \refL{index10}(iii), and \refEq{index12.19} to arrive at
\begin{equation}\tag{25}\labelp{Eq:index12.25}
 \ins{\tprod_i\abxy\xz s i}=\tprod_i \ins{\abxy\xz s i}=\tprod_i p_i^{\cs si}=\gcd\pair{\map xy}{\map yz}\per
\end{equation}

The indices of $\axy\xy$ and $\axy\yz$ are $\map xy$ and $\map yz$
respectively, by \refEq{index12.4}. Use this observation and
\refL{index1} to write
\begin{equation}\tag{26}\labelp{Eq:index12.26}
\gcd\pair{\map xy}{\map yz}=\gcd\pair{\ins{\axy\xy}}{\ins{\axy\yz}}=\ins{\axy\xy;\axy\yz}\per
\end{equation}
Combine \refEq{index12.26} with \refEq{index12.25} to arrive at
\begin{equation}\tag{27}\labelp{Eq:index12.27}
\ins{\axy\xy;\axy\yz}= \ins{\tprod_i\abxy\xz s i}\per
\end{equation} Use \refEq{index12.3}, \refEq{index12.17}, monotony, \refEq{index12.18},
and \refL{index10}(vi) to obtain
\begin{equation*}\tag{28}\labelp{Eq:index12.28}
\axy\xy;\axy\yz=(\tprod_{i}  \abxy\yz k i);(\tprod_{i} \abxy\yz \ell
i)\le \tprod_{i} ( \abxy\yz k i; \abxy\yz \ell i)\le \tprod_{i}
\abxy\yz s i\per
\end{equation*}
In view of \refEq{index12.28},   \refL{index2} (with
$\axy\xy;\axy\yz$ and $\tprod_{i}  \abxy\yz s i$ in place of $a$ and
$b$ respectively) may be applied to \refEq{index12.27}, and then
\refEq{index12.23} may be invoked, to conclude that
\begin{equation*}
\axy\xy;\axy\yz =  \tprod_i\abxy\xz s i \ge\axy\xz\per
\end{equation*}  This completes the verification of the third scaffold
condition and hence the proof of the theorem.
\end{proof}

Here, finally, is the representation theorem for measurable relation
algebras with finite cyclic groups.

\begin{theorem}[Representation
Theorem]\labelp{T:rep} If $\f A$ is a measurable relation
algebra\comma and if\comma for each measurable atom $x$\comma the
group $\G x$ is finite and cyclic\comma then $\f A$ is essentially
isomorphic to one of the cyclic group relation algebras constructed
in GCD Theorem \textnormal{\ref{T:gcd}}\per Hence, $\f A$ is
completely representable.
\end{theorem}

\begin{proof} Here is a summary of the strategy that was outlined at the beginning of the section.
The groups $\G x$ are all assumed to be finite, so $\f A$ is
finitely measurable and therefore atomic. Using this fact, a
scaffold is constructed in  $\f A$, by Scaffold \refT{scaffold}.  A
measurable relation algebra with a scaffold is essentially
isomorphic to a full group relation algebra, by Scaffold
Representation Theorem 7.4 in \cite{givand3}.

To see which full group relation algebra, let $I$ be the set of
measurable atoms in $\f A$, and for each atom $x$ in $I$, let $\G x$
be the group of non-zero functions below the square $x;1;x$.  Take
$\mc E$ to be the set of pairs $\pair xy$ such that $x;1;y\neq 0$.
Fix a scaffold
\begin{equation}\tag{1}\labelp{Eq:index13.1}
\langle\axy\xy:\pair xy\in \mc E\rangle
\end{equation} in $\f A$.  For each pair $\pair xy$ in $\mc E$, take $\h\xy$ and $\k\xy$
to be the left and right stabilizers of
the atom $\axy\xy$\per  The function $\vphi\xy$ from $\gll \xy$ to
$\grr\xy$ defined for cosets $H$ of $\h\xy$ and $K$ of $\k\xy$ by
\[\vphi\xy(H)=K\qquad\text{if and only if}\qquad
H;\axy\xy=\axy\xy;K\] is an isomorphism.  The group pair $\mc
F=\pair G\vp$ consisting of the systems
\[G=\langle\G x:x\in I\rangle\qquad\text{and}\qquad \vp=\langle\vphi\xy:\pair xy \in \mc E\rangle\]
satisfies the group frame conditions, by Frame Theorem 7.3 in \cite{givand3},
and $\f A$ is essentially isomorphic to the full group relation algebra $\cra G F$,
by (the proof of) Theorem 7.4 in \cite{givand3}.  The group relation algebra $\cra GF$ is one of
the ones considered in \refT{3.9}, and therefore also in  GCD \refT{gcd} (up to isomorphism).

Alternatively, let $\map xy=|\gll\xy|$ and observe that the system
\[m=\langle\map xy:\pair xy\in E\rangle\] satisfies the index conditions by the proof of
Scaffold \refT{scaffold}.  Consequently, $\f A$ is essentially
isomorphic to the cyclic group relation algebra constructed in GCD
\refT{gcd} using the system $m$.
\end{proof}

\renewcommand\di{0\textnormal{\rq}}
\newcommand{\Aa}{\mathfrak{A}}

Theorem 4.30 in J\'onson-Tarski \cite{jt52} states that for a
relation algebra $\Aa$ the following are equivalent. (i) $\Aa$ is
isomorphic to a full set relation algebra. (ii) $\Aa$ is complete,
atomic, with all atoms $x$ satisfying $x\conv;1;x\le\ident$. The
hard part of this theorem is to show that (ii) implies (i). Assume
(ii) and let $x$ be a subidentity atom. Then $x\conv;1;x\le\ident$
by assumption, and thus $(x;1;x)\conv;(x;1;x)=x\conv;1;x\le\ident$
by \refL{laws}(v) and \refL{square}(i),(ii). This means that $x;1;x$
is functional. Thus the square $x;1;x$ is the sum of one functional
element, hence $x$ is measurable with measure 1. Since $\Aa$ is
atomic, the identity is a sum of atoms, and we have seen that these
atoms are measurable, hence $\Aa$ is measurable. Each of the
associated groups have one element, thus finite and cyclic. Since
$\Aa$ is complete, then $\Aa$ is isomorphic, and not just
essentially isomorphic, to a group relation algebra $\cra G F$ with
all the groups in $\mc F$ being one-element. It is not hard to see
that such a $\cra G F$ is isomorphic to a full set relation algebra.
We have proved the hard part of \cite[Theorem 4.30]{jt52} by using
\refT{rep}.

We note that a representation theorem is given in \cite{agmns} which
uses a generalization of the above condition (ii) in another
direction, not toward measurability.

Next we turn to pair-dense relation algebras.
Let $\Aa$ be a relation algebra. In \cite{ma91}, an element $x\in A$
is called a \emph{pair} if $x;\di;x;\di;x\le\ident$ and $x$ is
nonzero, where $\di$ denotes $-\ident$, and the algebra $\Aa$ is
called \emph{pair-dense} if the identity element $\ident$ is a sum
of pairs.

\begin{lm}\label{pair-lem}
Let $\Aa$ be an atomic relation algebra. Then (i) and (ii) below are
equivalent.
\begin{enumerate}
\item[(i)] $\Aa$ is pair-dense.
\item[(ii)] $\Aa$ is measurable with all the associated groups cyclic of order $\le 2$.
\end{enumerate}
\end{lm}

\begin{proof}
In the proof, we will use \refL{laws}-\refL{square} without
mentioning them. Let $x\le\ident$ be an atom, in particular $x$ is
nonzero. First we show that $x$ is a pair just in case it is
measurable with measure $\le 2$. Now, by definition, $x$ is a pair
just in case $x;\di;x;\di;x\le\ident$, we are going to show that
this latter holds just when $x;\di;x$ is functional. Indeed,
$x;\di;x$ is functional, by definition, just when
$(x;\di;x)\conv;(x;\di;x)\le\ident$, but $(x;\di;x)\conv;(x;\di;x) =
(x;\di;x);(x;\di;x) = x;\di;x;\di;x$. Note that each subidentity
element $x$ is functional by $x\conv;x=x\le\ident$.

Assume now that $x$ is a pair. Then $x;1;x$ is the sum of
$x;\ident;x$ and $x;\di;x$, both being functional (since
$x;\ident;x=x$). The first one, $x;\ident;x$ is nonzero by $x$ being
nonzero. If the second one, $x;\di;x$ is zero, then $x$ has measure
1, and if $x;\di;x$ is nonzero, then $x$ has measure 2. We have seen
that $x$ is measurable with measure $\le 2$.

Assume now that $x$ has measure $\le 2$. Then $x$ is the sum of $\le
2$ nonzero functional elements, and we mentioned at the beginning of
\refS{sec4} that each of these functional elements is an atom. Since
$x$ is a subidentity atom, it is nonzero and functional. Now,
$x;1;x=x;\ident;x + x;\di;x$, where $x=x;\ident;x$. Thus, one of the
functional elements below $x;1;x$ is $x$ itself. If $x$ has measure
1, then $x;1;x=x$ and so $x;\di;x=0$ hence functional. If $x$ has
measure 2, then the other functional element below $x;1;x$ must be
$x;\di;x$. So in both cases $x;\di;x$ is functional. We have already
seen that $x;\di;x$ is functional just in case $x$ is a pair.

We are ready to prove the lemma. Assume that $\Aa$ is pair-dense and
atomic. By definition and monotony, if $x$ is a pair, then each
nonzero element below it is also a pair. Therefore, $\ident$ is the
sum of pairs that are atoms. We have seen that all these atoms are
measurable with measure $\le 2$, hence (ii) holds. Assume now that
(ii) holds, then $\ident$ is the sum of measurable atoms with
measure $\le 2$. Each of these atoms is a pair, so $\Aa$ is
pair-dense.

Finally, notice that each group of order $\le 2$ is cyclic.
\end{proof}

In view of Lemma~\ref{pair-lem}, our \refT{rep} implies that all
atomic pair-dense relation algebras are completely representable,
and in fact essentially isomorphic to a group relation algebra where
the associated groups have order one or two. This gives a structural
description for atomic pair-dense relation algebras. Theorem 48 of
\cite{ma91} states that simple pair-dense relation algebras are
atomic, and Theorem 51 of \cite{ma91} states that for simple
pair-dense algebras $(\alpha)$ and $(\beta)$ are equivalent, where
$(\alpha)$ states that $\Aa$ is completely representable on the set
$U$, and $(\beta)$ states that the cardinality of $U$ is $n+2m$
where $n$ is the number of atomic pairs $x$ below $\ident$ for which
$x;\di;x$ is zero, and $m$ is the number for those where $x;\di;x$
is nonzero. Now, using \cite[Theorem 48]{ma91} and \refT{rep}, Lemma
\ref{pair-lem}, one can give an alternative proof for \cite[Theorem
51]{ma91}.

\end{document}